\documentclass[12pt,reqno]{article}

\usepackage[usenames]{color}
\usepackage{amssymb}
\usepackage{amsmath}
\usepackage{amsthm}
\usepackage{amsfonts}
\usepackage{amscd}
\usepackage{graphicx}

\usepackage[colorlinks=true,
linkcolor=webgreen,
filecolor=webbrown,
citecolor=webgreen]{hyperref}

\definecolor{webgreen}{rgb}{0,.5,0}
\definecolor{webbrown}{rgb}{.6,0,0}

\usepackage{color}
\usepackage{fullpage}
\usepackage{float}

\usepackage{graphics}
\usepackage{latexsym}

\begin{document}

\theoremstyle{plain}
\newtheorem{theorem}{Theorem}
\newtheorem{corollary}[theorem]{Corollary}
\newtheorem{lemma}{Lemma}
\newtheorem{example}{Example}
\newtheorem*{remark}{Remark}

\begin{center}
\vskip 1cm
{\LARGE\bf Fibonacci sums and divisibility properties \\ }

\vskip 1cm

{\large
Kunle Adegoke \\
Department of Physics and Engineering Physics \\ Obafemi Awolowo University, 220005 Ile-Ife, Nigeria \\
\href{mailto:adegoke00@gmail.com}{\tt adegoke00@gmail.com}

\vskip 0.2 in

Robert Frontczak \\
Independent Researcher \\ Reutlingen,  Germany \\
\href{mailto:robert.frontczak@web.de}{\tt robert.frontczak@web.de}

}

\end{center}

\vskip .2 in

\begin{abstract}
Based on a variant of Sury's polynomial identity we derive new expressions for various finite Fibonacci (Lucas) sums.
We extend the results to Fibonacci and Chebyshev polynomials, and also to Horadam sequences.
In addition to deriving sum relations, the main identities will be shown to be very useful in establishing and
discovering divisibility properties of Fibonacci and Lucas numbers.
\end{abstract}



\bigskip

\section{Introduction}

As usual, we will use the notation $F_n$ for the $n$th Fibonacci number and $L_n$ for the $n$th Lucas number, respectively.
Both number sequences are defined, for \text{$n\in\mathbb Z$}, through the same recurrence relation $x_n = x_{n-1} + x_{n-2}, n\ge 2,$
with initial values $F_0=0, F_1=1$, and $L_0=2, L_1=1$, respectively. For negative subscripts we have $F_{-n}=(-1)^{n-1}F_n$
and $L_{-n}=(-1)^n L_n$. They possess the explicit formulas (Binet forms)
\begin{equation*}
F_n = \frac{\alpha^n - \beta^n }{\alpha - \beta },\quad L_n = \alpha^n + \beta^n,\quad n\in\mathbb Z.
\end{equation*}
For more information about these famous sequences we refer, among others, to the books by Koshy \cite{Koshy} and Vajda \cite{Vajda}.
In addition, one can consult the On-Line Encyclopedia of Integer Sequences
\cite{OEIS} where these sequences are listed under the ids {A000045} and A000032, respectively.

In 2014, Sury \cite{Sury} presented a polynomial identity in two variables $u$ and $v$ of the form
\begin{equation}\label{sury}
(2u)^{n+1} - (2v)^{n+1} = (u-v) \sum_{j=0}^n ((2u)^j + (2v)^j)(u+v)^{n-j}.
\end{equation}
As it contains the relation
\begin{equation}\label{Fib_id}
2^{n+1} F_{n+1} = \sum_{j=0}^n 2^j L_j,
\end{equation}
as a special instance, Sury called \eqref{sury} a polynomial parent to \eqref{Fib_id}. It is obvious that identity~ \eqref{Fib_id}
can be derived directly using the geometric series. A slightly more general result is
\begin{align}
2^{n+1} F_{n+r+1} - F_r &= \sum_{j=0}^n 2^j L_{j+r}, \\
\frac{1}{5}\left (2^{n+1} L_{n+r+1} - L_r\right ) &= \sum_{j=0}^n 2^j F_{j+r},
\end{align}
or even
\begin{equation}
2^{n+1} G_{n+r+1} - G_r = \sum_{j=0}^n 2^j \left ( G_{j+r+1} + G_{j+r-1}\right ),
\end{equation}
where $r$ is an integer and $G_n$ is a Gibonacci sequence, i.e., a sequence given by $G_0=a$, $G_1=b$ and for $n\geq 2$, $G_{n}=G_{n-1}+G_{n-2}$.

In this paper, we apply a variant of Sury's polynomial identity to derive new expressions for various finite Fibonacci (Lucas) sums.
As will be seen immediately we can infer some divisibility properties for Fibonacci (Lucas) numbers from these sums.
Some divisibility properties were also studied by Hoggatt and Bergum~\cite{Hoggatt74} and in the recent articles by
Pongsriiam \cite{Pong} and Onphaeng and Pongsriiam \cite{On}, among others.
Extensions will be provided to Fibonacci (Lucas) polynomials, Chebyshev polynomials, and finally to Horadam sequences.

\section{First Results}

The next polynomial identity is a variant of Sury's identity and will be of crucial importance in this paper.

\begin{lemma}\label{main_lem}
If $x$ and $y$ are any complex variables and $n$ is any integer, then
\begin{equation}\label{eq.g961w92}
f(x,y) = \sum_{j = 0}^n {(xy)^j \left( {x^{n - 2j}  + y^{n - 2j} } \right)}
= \sum_{j = 0}^n {\left( {\frac{{x + y}}{2}} \right)^j \left( {x^{n - j}  + y^{n - j}} \right)}
\end{equation}
and
\begin{equation}\label{value}
f(x,y)=\frac{2(x^{n + 1} - y^{n + 1})}{x - y}.
\end{equation}
\end{lemma}

In addition to deriving sum relations, identities~\eqref{eq.g961w92} and~\eqref{value} are going to be very useful in establishing and discovering divisibility properties of Fibonacci and Lucas numbers.

\begin{theorem}\label{thm1}
If $r$ and $n$ are any integers, then
\begin{equation}
\sum_{j = 0}^n (- 1)^{rj} L_{r(n - 2j)} = \sum_{j = 0}^n \left( {\frac{{L_r }}{2}} \right)^j L_{r(n - j)} = \frac{2F_{r(n + 1)}}{F_r}.
\end{equation}
\end{theorem}
\begin{proof}
Set $x=\alpha^r$ and $y=\beta^r$ in~\eqref{eq.g961w92}, and use \eqref{value} and the Binet formulas.
\end{proof}
Theorem \ref{thm1} offers a new \color[rgb]{0,0,0}simple proof of a well-known fact concerning the divisibility of Fibonacci numbers.
\begin{corollary}
If $m$ and $r$ are integers, then $F_r$ divides $F_{mr}$.
\end{corollary}

\begin{theorem}\label{thm.wtn85im}
If $r$ and $n$ are any integers, then
\begin{equation}\label{eq.siz32hc}
\begin{split}
\sum_{j = 0}^{2n} (- 1)^{j(r + 1)} L_{2r(n - j)}  &= \sum_{j = 0}^n \left( {\frac{{F_r }}{2}} \right)^{2j} 5^j L_{2r(n - j)}
+ \sum_{j = 1}^n \left( {\frac{{F_r }}{2}} \right)^{2j - 1} 5^j F_{r(2n - 2j + 1)} \\
&= \frac{2L_{r(2n + 1)}}{L_r},
\end{split}
\end{equation}
and
\begin{equation}\label{eq.ousxx9b}
\begin{split}
\sum_{j = 0}^{2n - 1} (- 1)^{j(r + 1)} F_{r(2n - 2j - 1)} &= \sum_{j = 0}^{n - 1} \left( {\frac{{F_r }}{2}} \right)^{2j} 5^j F_{r(2n - 2j - 1)}+ \sum_{j = 1}^n \left( {\frac{{F_r }}{2}} \right)^{2j - 1} 5^{j - 1} L_{r(2n - 2j)} \\
&= \frac{2F_{2rn}}{L_r}.
\end{split}
\end{equation}
\end{theorem}
\begin{proof}
Write~\eqref{eq.g961w92} as
\begin{equation}\label{eq.jku4p1d}
\begin{split}
\sum_{j = 0}^{2n} {(xy)^j \left( {x^{2n - 2j}  + y^{2n - 2j} } \right)}
&= \sum_{j = 0}^n {\left( {\frac{{x + y}}{2}} \right)^{2j} \left( {x^{2n - 2j}  + y^{2n - 2j} } \right)} \\
&\qquad + \sum_{j = 1}^n {\left( {\frac{{x + y}}{2}} \right)^{2j - 1} \left( {x^{2n - 2j + 1}  + y^{2n - 2j + 1} } \right)} ;
\end{split}
\end{equation}
set $x=\alpha^r$ and $y=-\beta^r$ and combine according to the Binet formulas; thereby proving~\eqref{eq.siz32hc}.
To prove~\eqref{eq.ousxx9b}, write~\eqref{eq.g961w92} as
\begin{equation}\label{eq.vgfrsd}
\begin{split}
\sum_{j = 0}^{2n - 1} {(xy)^j \left( {x^{2n - 2j - 1}  + y^{2n - 2j - 1} } \right)}  &= \sum_{j = 0}^{n - 1} {\left( {\frac{{x + y}}{2}} \right)^{2j} \left( {x^{2n - 2j - 1}  + y^{2n - 2j - 1} } \right)} \\
&\qquad + \sum_{j = 1}^n {\left( {\frac{{x + y}}{2}} \right)^{2j - 1} \left( {x^{2n - 2j}  + y^{2n - 2j} } \right)} ;
\end{split}
\end{equation}
and set $x=\alpha^r$ and $y=-\beta^r$. Finish the proof in both cases using \eqref{value}.
\end{proof}

\begin{corollary}
If $m$ is an odd integer, then $L_r$ divides $L_{mr}$. Also, if $m$ is an even integer, then $L_r$ divides $F_{mr}$.
\end{corollary}

\begin{theorem}\label{thm.ksm9y59}
If $r$ and $n$ are any integers, then
\begin{equation}
\begin{split}
& \sum_{j = 0}^n {F_r^j F_{r - 1}^{n - j} L_j }
= \sum_{j = 0}^n \frac{1}{2^{j - 1}} \left( {F_r^{n - j} L_{n + j(r - 1)} + F_{r - 1}^{n - j} L_{rj} } \right) \\
& \qquad \qquad =  \frac{F_r^{n+2} L_n + F_{r-1} F_{r}^{n+1} L_{n+1} + F_r F_{r-1}^{n+1} - 2 F_{r-1}^{n+2}}{F_r^2 + F_r F_{r-1} - F_{r-1}^2}
\end{split}
\end{equation}
and
\begin{equation}
\begin{split}
& \sum_{j = 0}^n {F_r^j F_{r - 1}^{n - j} F_j }
= \sum_{j = 0}^n \frac{1}{2^{j - 1}} \left( {F_r^{n - j} F_{n + j(r - 1)} + F_{r - 1}^{n - j} F_{rj}} \right) \\
& \qquad \qquad =  \frac{F_r^{n+2} F_n + F_{r-1} F_{r}^{n+1} F_{n+1} - F_r F_{r-1}^{n+1}}{F_r^2 + F_r F_{r-1} - F_{r-1}^2}.
\end{split}
\end{equation}
\end{theorem}
\begin{proof}
Set $(x,y)=(\alpha F_r,F_{r - 1})$ and $(x,y)=(\beta F_r,F_{r - 1})$, in turn, in \eqref{eq.g961w92} and \eqref{value}, respectively.
Use the relations
\begin{equation*}
\alpha F_r + F_{r - 1} = \alpha^r \qquad\mbox{and}\qquad \beta F_r + F_{r - 1} = \beta^r
\end{equation*}
and combine using the Binet formulas.
\end{proof}

\begin{corollary}
If $r$ is a non-zero integer and $n$ is any positive integer, then
\begin{align*}
F_r^2 + F_r F_{r-1} - F_{r-1}^2 \mid F_r^{n+2} L_n + F_{r-1} F_{r}^{n+1} L_{n+1} + F_r F_{r-1}^{n+1} - 2 F_{r-1}^{n+2} ,\\
F_r^2 + F_r F_{r-1} - F_{r-1}^2 \mid F_r^{n+2} F_n + F_{r-1} F_{r}^{n+1} F_{n+1} - F_r F_{r-1}^{n+1}.
\end{align*}
\end{corollary}
In particular,
\begin{align*}
5 \mid (2^{n + 1}L_{n+1} - 2),\\
11 \mid (3^{n + 1}(F_{n+2} + L_{n+1}) - 3\cdot 2^{n + 1}),\\
11 \mid (3^{n + 1}(L_{n+2} + 5F_{n+1}) - 2^{n + 1}).
\end{align*}

\begin{theorem}\label{thm.cr7gcas}
If $r$ and $n$ are any integers, then
\begin{equation}
\begin{split}
2\sum_{j = 0}^n (-1)^{r(n - j)} L_{2rj}
& = \sum_{j = 0}^n \Big ( \frac{L_r}{2}\Big )^j \left( L_{r(2n - j)} + (-1)^{r(n-j)} L_{rj} \right)  \\
& = 2 \frac{(-1)^{r+1} L_{2r(n+1)} - (-1)^{r(n+1)} L_{2r} + L_{2rn} + 2(-1)^{rn}}{(-1)^{r+1} 5 F_r^2}
\end{split}
\end{equation}
and
\begin{equation}
\begin{split}
2\sum_{j = 0}^n (-1)^{r(n - j)} F_{2rj}
& = \sum_{j = 0}^n \Big ( \frac{L_r}{2}\Big )^j \left( F_{r(2n - j)} + (-1)^{r(n-j)} F_{rj} \right) \\
& = 2 \frac{(-1)^{r+1} F_{2r(n+1)} + (-1)^{r(n+1)} F_{2r} + F_{2rn}}{(-1)^{r+1} 5 F_r^2}.
\end{split}
\end{equation}
\end{theorem}
\begin{proof}
Set $(x,y)=(\alpha^{2r},(-1)^r)$ and $(x,y)=(\beta^{2r},(-1)^r)$, in turn, in \eqref{eq.g961w92} and \eqref{value}, respectively.
Use the relations
\begin{equation*}
\alpha^{2r} + (-1)^r = \alpha^r L_r \qquad\mbox{and}\qquad \beta^{2r} + (-1)^r = \beta^r L_r
\end{equation*}
and combine using the Binet formulas.
\end{proof}

\begin{corollary}
If $r$ is a non-zero integer and $n$ is any positive integer, then
\begin{align*}
5 F_r^2 \mid (-1)^{r+1} F_{2r(n+1)} + (-1)^{r(n+1)} F_{2r} + F_{2rn}.
\end{align*}
\end{corollary}
In particular,
\begin{equation*}
5 \mid (F_{2(n+1)} + F_{2n} + (-1)^{n+1}) = L_{2n + 1} + (-1)^{n + 1}.
\end{equation*}

\begin{remark}
We also have that
\begin{align*}
5 F_r^2 \mid (-1)^{r+1} L_{2r(n+1)} - (-1)^{r(n+1)} L_{2r} + L_{2rn} + 2(-1)^{rn}.
\end{align*}
But this is obvious as
\begin{equation*}
(-1)^{r+1} L_{2r(n+1)} - (-1)^{r(n+1)} L_{2r} + L_{2rn} + 2(-1)^{rn} = (-1)^{r+1} 5 F_r F_{2rn+r} - (-1)^{r(n+1)} 5 F_r^2
\end{equation*}
and because $F_r|F_{r(2n+1)}$.
\end{remark}

\begin{theorem}
If $n$ and $r$ are any integers, then
\begin{eqnarray}
\sum_{j = 0}^n {L_{2r}^j 2^{n - j + 1} } & = &
\begin{cases}
 \sum_{j = 0}^n {\left( {\frac{{L_r^2 }}{2}} \right)^j\left( {L_{2r}^{n - j}  + 2^{n - j} } \right)},&\mbox{if $r$ is even};  \\
 \sum_{j = 0}^n {\left( {\frac{{5F_r^2 }}{2}} \right)^j\left( {L_{2r}^{n - j}  + 2^{n - j} } \right)},&\mbox{if $r$ is odd};  \\
\end{cases} \nonumber \\
& = & \begin{cases}
2(L_{2r}^{n + 1}  - 2^{n + 1} )/(5F_r^2 ), & \mbox{if $r$ is even};  \\
2(L_{2r}^{n + 1}  - 2^{n + 1} )/L_r^2, & \mbox{if $r$ is odd}.  \\
\end{cases}
\end{eqnarray}
\end{theorem}
\begin{proof}
Set $x=L_{2r}$ and $y=2$ in \eqref{eq.g961w92} and \eqref{value}, respectively, and use
\begin{equation*}
L_{2r} + 2
=  \begin{cases}
 5F_r^2 , \quad\mbox{$r$ odd};\\
 L_r^2 , \quad\mbox{$r$ even}.\\
 \end{cases}\label{eq.e10mht5}
\end{equation*}
\end{proof}
\begin{corollary}
If $r$ is a non-zero integer and $n$ is any positive integer, then
\begin{align*}
5F_r^2\mid (L_{2r}^{n + 1} - 2^{n + 1}),&\text{ if $r$ is even},\\
L_r^2\mid (L_{2r}^{n + 1} - 2^{n + 1}),&\text{ if $r$ is odd}.
\end{align*}
\end{corollary}

\begin{theorem}
If $r$ and $n$ are any integers, then
\begin{equation}
\begin{split}
2\sum_{j = 0}^n (- 1)^{rj} 4^j F_r^{2n - 2j} 5^{n - j} &= \sum_{j = 0}^n {\left( {\frac{{L_r^2 }}{2}} \right)^j\left( {5^{n - j} F_r^{2(n - j)}  + ( - 1)^{r(n - j)} 4^{n - j} } \right)} \\
&= 2 \frac{(5F_r^2)^{n+1} - (-1)^{r(n+1)} 4^{n+1}}{5F_r^2 - (-1)^r 4}.
\end{split}
\end{equation}
\end{theorem}
\begin{proof}
Set $x=5F_r^2$ and $y=(-1)^r4$ in \eqref{eq.g961w92} and \eqref{value}, respectively, and use the identity
\begin{equation*}
5F_r^2 + (-1)^r4=L_r^2.
\end{equation*}
\end{proof}

\begin{theorem}
If $r$, $n$ and $t$ are any integers, then
\begin{equation}
\begin{split}
\sum_{j = 0}^{2n} {L_r^j L_{r - 1}^{2n - j}L_{j + t}}&= \sum_{j = 0}^n {\frac{{5^j }}{{2^{2j + 1} }}\left( {L_r^{2n - 2j} L_{2n - 2j + 2jr + t}  + L_{r - 1}^{2n - 2j} L_{2jr + t} } \right)} \\
&\qquad\quad + \sum_{j = 1}^n {\frac{{5^j }}{{2^{2j} }}\left( {L_r^{2n - 2j + 1} F_{2n - 2j + (2j - 1)r + t}  + L_{r - 1}^{2n - 2j + 1} F_{(2j - 1)r + t} } \right)} \\
& = \frac{{L_r^{2n + 1} \left( {L_r L_{2n + t}  + L_{r - 1} L_{2n + t + 1} } \right) - L_{r - 1}^{2n + 1} \left( {L_r L_{t - 1}  + L_{r - 1} L_t } \right)}}{{L_{r - 2} L_{r + 1}  + L_r L_{r - 1} }}
\end{split}
\end{equation}
and
\begin{equation}
\begin{split}
\sum_{j = 0}^{2n} {L_r^j L_{r - 1}^{2n - j}F_{j + t}}&= \sum_{j = 0}^n {\frac{{5^j }}{{2^{2j + 1} }}\left( {L_r^{2n - 2j} F_{2n - 2j + 2jr + t}  + F_{r - 1}^{2n - 2j} F_{2jr + t} } \right)} \\
&\qquad\quad + \sum_{j = 1}^n {\frac{{5^j }}{{2^{2j} }}\left( {L_r^{2n - 2j + 1} L_{2n - 2j + (2j - 1)r + t}  + L_{r - 1}^{2n - 2j + 1} L_{(2j - 1)r + t} } \right)}\\
& = \frac{{L_r^{2n + 1} \left( {L_r F_{2n + t}  + L_{r - 1} F_{2n + t + 1} } \right) - L_{r - 1}^{2n + 1} \left( {L_r F_{t - 1}  + L_{r - 1} F_t } \right)}}{{L_{r - 2} L_{r + 1}  + L_r L_{r - 1} }}.
\end{split}
\end{equation}
\end{theorem}
\begin{proof}
Set $x=\alpha L_r$ and $y=L_{r -1}$ in~\eqref{eq.jku4p1d}, noting that
\begin{equation}
\alpha L_r + L_{r - 1}=\alpha^r\sqrt 5.
\end{equation}
Multiply through the resulting equation by $\alpha^t$. Use $2\alpha^s=L_s + F_s\sqrt 5$ to reduce the resulting equation.
Finally, compare the coefficients of $\sqrt 5$.
\end{proof}

\begin{corollary}
If $r$, $n$ and $t$ are any integers, then
\begin{gather}
L_{r - 2} L_{r + 1}  + L_r L_{r - 1} \mid L_r^{2n + 1} \left( {L_r L_{2n + t}  + L_{r - 1} L_{2n + t + 1} } \right)
- L_{r - 1}^{2n + 1} \left( {L_r L_{t - 1}  + L_{r - 1} L_t } \right),\label{eq.fs9p2xj}\\
L_{r - 2} L_{r + 1}  + L_r L_{r - 1} \mid L_r^{2n + 1} \left( {L_r F_{2n + t}  + L_{r - 1} F_{2n + t + 1} } \right)
- L_{r - 1}^{2n + 1} \left( {L_r F_{t - 1}  + L_{r - 1} F_t } \right)\label{eq.wl74ip3}.
\end{gather}
In particular,
\begin{align*}
11&\mid3^{2n + 1} \left( {L_{2n}  + 5F_{2n + 1} } \right) + 1,\\
11&\mid3^{2n + 1} \left( {F_{2n}  + L_{2n + 1} } \right) - 3.
\end{align*}
\end{corollary}
\begin{theorem}
If $r,k,s$ and $n$ are any integers, then
\begin{equation}
\begin{split}
2\sum_{j = 0}^n (- 1)^{(k+s)j} L_{r-s}^{j} L_{2k+r+s}^{n - j}
&= \sum_{j = 0}^n \left( {\frac{L_{k+r}L_{k+s}}{2}} \right)^j \left( L_{2k+r+s}^{n - j}  + (- 1)^{(k+s)(n - j)} L_{r-s}^{n - j} \right) \\
&= 2 \frac{L_{2k+r+s}^{n+1} - (- 1)^{(k+s)(n + 1)} L_{r-s}^{n + 1}}{5 F_{k+r} F_{k+s}}.
\end{split}
\end{equation}
\end{theorem}
\begin{proof}
Set $x=L_{2k+r+s}$ and $y=(- 1)^{k+s} L_{r-s}$ in \eqref{eq.g961w92} and \eqref{value}, respectively, and use the identities \cite{Vajda}
\begin{align*}
L_{2k+r+s} + (- 1)^{k+s} L_{r-s} = L_{k+r} L_{k+s}, \\
L_{2k+r+s} - (- 1)^{k+s} L_{r-s} = 5 F_{k+r} F_{k+s}.
\end{align*}
\end{proof}

\begin{corollary}
If $r,k,s$ are integers and $n$ is any non-negative integer, then
\begin{equation}
5 F_{k+r} F_{k+s} \mid (L_{2k+r+s}^{n + 1} - (- 1)^{(k+s)(n + 1)} L_{r-s}^{n + 1}).
\end{equation}
In particular,
\begin{equation}
5 F_{k+r}^2 \mid (L_{2(k+r)}^{n + 1} - (- 1)^{(k+r)(n + 1)} 2^{n + 1}).
\end{equation}
\end{corollary}

\section{Extension to Fibonacci polynomials}

Fibonacci (Lucas) polynomials are polynomials that can be defined by the Fibonacci-like recursion
and generalizing Fibonacci (Lucas) numbers. They were already studied in 1883 by E. Catalan and E. Jacobsthal.
For any integer $n\geq0$, the Fibonacci polynomials $\{F_n(x)\}_{n\geq0}$ are defined by the second-order recurrence relation
\begin{equation*}
F_0(x)=0, \quad F_1(x)=1, \quad F_{n+1}(x) = xF_{n}(x) + F_{n-1}(x),
\end{equation*}
while the Lucas polynomials $\{L_n(x)\}_{n\geq0}$ follow the rule
\begin{equation*}
L_0(x)=2, \quad L_1(x)=x, \quad L_{n+1}(x) = xL_{n}(x) + L_{n-1}(x).
\end{equation*}
Their Binet forms are given by
\begin{equation*}
F_{n}(x) = \frac{\alpha^n(x) - \beta^n(x)}{\alpha(x) - \beta(x)}, \qquad L_n(x) = \alpha^n(x) + \beta^n(x),
\end{equation*}
where $\alpha(x)=\frac{x+\sqrt{x^2+4}}{2}$ and $\beta(x)=\frac{x-\sqrt{x^2+4}}{2}$.

\begin{theorem}\label{Fibpol_thm1}
For any non-negative integer $n$ we have
\begin{equation}\label{Fibpol_id1}
\sum_{j=0}^n (-1)^j L_{n-2j}(x) = \sum_{j=0}^n \Big (\frac{x}{2}\Big )^j L_{n-j}(x) = 2 F_{n+1}(x).
\end{equation}
\end{theorem}
\begin{proof}
Apply Lemma \ref{main_lem} inserting $x=\alpha(x)$ and $y=\beta(x)$.
\end{proof}

\begin{corollary}
For any non-negative integer $n$ we have
\begin{equation}
\sum_{j=0}^n (-1)^j Q_{n-2j} = \sum_{j=0}^n Q_{n-j} = 2 P_{n+1},
\end{equation}
where $P_n=F_n(2)$ and $Q_n=L_n(2)$ are the Pell and Pell-Lucas numbers, respectively.
\end{corollary}
\begin{proof}
Insert $x=2$ and use $F_n(2)=P_n$ and $L_n(2)=Q_n$, respectively.
\end{proof}

\begin{remark}
We mention that a different proof of Theorem \ref{thm1} can be provided by inserting $x=L_r$, $r$ odd, and $x=i L_r$, $r$ even,
in Theorem \ref{Fibpol_thm1} and making use of
\begin{equation*}
L_n(L_r) = L_{rn}, \qquad F_n(L_r) = \frac{F_{rn}}{F_r}, \qquad r \,\,\mbox{odd},
\end{equation*}
and
\begin{equation*}
L_n(i L_r) = i^n L_{rn}, \qquad F_n(i L_r) = i^{n-1} \frac{F_{rn}}{F_r}, \qquad r \,\,\mbox{even}.
\end{equation*}
\end{remark}

\begin{theorem}\label{Fibpol_thm2}
For any non-negative integer $n$ and any $x\neq 0$ we have
\begin{equation}\label{Fibpol_id2}
\sum_{j=0}^n L_{2(n-2j)}(x) = \sum_{j=0}^n \Big (\frac{x^2+2}{2}\Big )^j L_{2(n-j)}(x) = \frac{2}{x} F_{2(n+1)}(x).
\end{equation}
\end{theorem}
\begin{proof}
Apply Theorem \ref{Fibpol_thm1} with $x=i(x^2+1),i=\sqrt{-1},$ and use
\begin{equation*}
F_n(i(x^2+1)) = i^{n-1} \frac{F_{2n}(x)}{x}\qquad \mbox{and}\qquad  L_n(i(x^2+1)) = i^n L_{2n}(x).
\end{equation*}
\end{proof}

\begin{theorem}\label{Fibpol_thm3}
For any non-negative integer $n$, any positive integer $r$, and any $x\neq 0$ we have
\begin{equation}\label{Fibpol_id3}
2 \sum_{j=0}^n F_{r-1}^j(x) F_{r+1}^{n-j}(x)
= \sum_{j=0}^n \Big (\frac{L_r(x)}{2}\Big )^j \left ( F_{r+1}^{n-j}(x) + F_{r-1}^{n-j}(x)\right )
= 2 \frac{F_{r+1}^{n+1}(x) - F_{r-1}^{n+1}(x)}{x F_r(x)} .
\end{equation}
\end{theorem}

\begin{corollary}
For any $n\geq 0$ and $m\geq 1$ we have
\begin{equation}
F_m^n L_m F_{rm} \mid F_{m(r+1)}^{n+1} - F_{m(r-1)}^{n+1}, \qquad m \,\,\mbox{odd},
\end{equation}
and
\begin{equation}
F_m^n L_m F_{rm} \mid F_{m(r+1)}^{n+1} + (-1)^n F_{m(r-1)}^{n+1}, \qquad m \,\,\mbox{even}.
\end{equation}
\end{corollary}

\section{Extension to Chebyshev polynomials}

Recall that, for any integer $n\geq0$, the Chebyshev polynomials $\{T_n(x)\}_{n\geq0}$ of the first kind are defined
by the second-order recurrence relation \cite{Mason}
\begin{equation}\label{T-def}
T_0(x)=1,\quad T_1(x)=x,\quad T_{n+1}(x) =2x T_n(x) - T_{n-1}(x),
\end{equation}
while the Chebyshev polynomials $\{U_n(x)\}_{n\geq0}$ of the second kind are defined by
\begin{equation}\label{U-def}
U_0(x)=1,\quad U_1(x)=2x,\quad U_{n+1}(x) = 2xU_n(x) - U_{n-1}(x).
\end{equation}
The sequences $T_n(x)$ and $U_n(x)$ have the exact (Binet) formulas
\begin{gather}
T_n (x) = \frac{1}{2} \left( (x + \sqrt {x^2 - 1} )^n + (x - \sqrt {x^2 - 1} )^n \right), \\
U_n (x) = \frac{1}{2\sqrt {x^2  - 1}} \left( (x + \sqrt {x^2 - 1} )^{n + 1} - (x - \sqrt {x^2 - 1} )^{n + 1} \right).
\end{gather}
More information about these polynomials can be found the book by Mason and Handscomb~\cite{Mason} and also in the recent articles by Frontczak and Goy \cite{Frontczak}, Fan and Chu \cite{Fan} and Adegoke et al. \cite{Adegoke}.

\begin{theorem}
For any integer $n$ we have
\begin{equation}\label{chebyshev_id}
\sum_{j=0}^n T_{n-2j}(x) = \sum_{j=0}^n x^j T_{n-j}(x) = U_n(x).
\end{equation}
\end{theorem}
\begin{proof}
Apply Lemma \ref{main_lem} inserting $x\mapsto x+\sqrt{x^2+1}$ and $y\mapsto x-\sqrt{x^2+1}$.
\end{proof}
Although \eqref{chebyshev_id} offers a very appealing relation we have learnt that it is not new.
It was proved by completely other methods in 1985 by Boscarol \cite{Boscarol}.

\section{Extension to the Horadam sequence}

Lemma \ref{main_lem} in the form
\begin{equation}\label{eq.ttxukxg}
f(x,y) = \sum_{j = 0}^n {x^j y^{n - j} } = \sum_{j = 0}^n \frac{{(x + y)^j }}{{2^{j + 1} }}\left( {x^{n - j}  + y^{n - j} } \right)
= \frac{{x^{n + 1} - y^{n + 1} }}{{x - y}}
\end{equation}
readily allows sum relations to be derived for the Horadam sequence and divisibility properties to be established.

Let $\{w_n(a,b;p,q)\}_{n\geq 0}$ be the Horadam sequence \cite{horadam65} defined for all non-negative integers~$n$ by the recurrence
\begin{equation}\label{eq.vhrb5b3}
w_0 = a,\,\,\,w_1 = b;\quad w_n = pw_{n - 1} - qw_{n - 2},\quad n \ge 2,
\end{equation}
where $a$, $b$, $p$ and $q$ are arbitrary complex numbers, with $p\ne 0$ and $q\ne 0$. Extension of the definition of $w_n(a,b;p,q)$ to negative subscripts is provided by writing the recurrence relation as
$$w_{-n} = \frac{1}{q} (p w_{-n+1} - w_{-n+2})$$
where, for brevity, we wrote (and will write) $w_n$ for $w_n(a,b;p,q)$. \\

Two important cases of $w_n$ are the Lucas sequences of the first kind, $u_n(p,q)=w_n(0,1;p,q)$,
and of the second kind, $v_n(p,q)=w_n(2,p;p,q)$. The most well-known Lucas sequences are the Fibonacci sequence $F_n=u_n(1,-1)$ and the sequence of Lucas numbers $L_n=v_n(1,-1)$.

The Binet formulas for sequences $u_n$, $v_n$ and $w_n$ in the non-degenerate case, $p^2 - 4q > 0$, are
\begin{equation}\label{bine.uvw}
u_n = \frac{\tau^n - \sigma^n}{\sqrt{p^2 - 4q}}=\frac{\tau^n - \sigma^n}{\Delta},\qquad v_n = \tau^n + \sigma^n, \qquad w_n = A\tau^n + B\sigma^n,
\end{equation}
with
$\displaystyle A=\frac{b - a\sigma}{\Delta}$ and $B\displaystyle=\frac{a\tau  - b}{\Delta}$,
where
$$
\tau = \tau(p,q) = \frac{p+\Delta}{2}\text{ and } \sigma = \sigma(p,q) = \frac{p-\Delta}{2}
$$
are the distinct zeros of the characteristic polynomial $x^2-px+q$ of the Horadam sequence.

In this section, we will make use of the following known results.
\begin{lemma}\label{lem.jv8c0fd}
If $a$, $b$, $c$ and $d$ are rational numbers and $\lambda$ is an irrational number, then
\begin{equation*}
a + b\,\lambda=c + d\,\lambda\quad \iff\quad a=c,\,\,\, b=d.
\end{equation*}
\end{lemma}
\begin{lemma}\label{lem.wiwyib1}
For any integer $s$,
\begin{align}
&q^s + \tau ^{2s} = \tau ^s v_s, \qquad q^s - \tau ^{2s} = - \Delta\tau ^s u_s, \label{eq.u9uuagc} \\
&q^s + \sigma ^{2s} = \sigma ^s v_s, \qquad q^s - \sigma ^{2s} = \Delta\sigma ^s u_s. \label{eq.dkvfyiz}
\end{align}
In particular,
\begin{align}
&( - 1)^s + \alpha ^{2s} = \alpha ^s L_s,\qquad  ( - 1)^s - \alpha ^{2s} = - \sqrt 5\alpha ^s F_s, \label{eq.ozz3zp6} \\
&( - 1)^s + \beta ^{2s} = \beta ^s L_s,\qquad ( - 1)^s - \beta ^{2s} = \sqrt 5\beta ^s F_s \label{eq.syrjcay}.
\end{align}
\end{lemma}
\begin{lemma}\label{lem.ydalnfx}
Let $r$ and $s$ be any integers. Then
\begin{align}
&v_{r + s} - \tau ^r v_s  =  - \Delta\sigma ^s  u_r \label{eq.j428hfx},\\
&v_{r + s} - \sigma ^r v_s  = \Delta \tau ^s u_r \label{eq.cqli6xc},\\
&u_{r + s} - \tau ^r u_s  = \sigma ^s u_r\label{eq.iamiky1},\\
&u_{r + s} - \sigma ^r u_s  = \tau ^s u_r\label{eq.zy0gfyn}.
\end{align}
In particular, \cite{Hoggatt},
\begin{align}
&L_{r + s} - L_r \alpha ^s = - \sqrt 5 \beta ^r F_s,\qquad  L_{r + s} - L_r \beta ^s = \sqrt 5 \alpha ^r F_s, \label{es.benssj4} \\
&F_{r + s} - F_r \alpha ^s = \beta ^r F_s,\qquad\qquad\,\, F_{r + s} - F_r \beta ^s = \alpha ^r F_s. \label{eq.pvdw5ja}
\end{align}
\end{lemma}
\begin{lemma}\label{lem.w65xm59}
For any integer $n$,
\begin{align}
&A\tau ^n - B\sigma ^n = \frac{{w_{n + 1} - qw_{n - 1} }}{\Delta }\label{eq.p951mmh},\\
&A\sigma ^n + B\tau ^n = q^n w_{-n}\label{eq.trp91mj}.
\end{align}
\end{lemma}
\begin{proof}
See \cite[Lemma 1]{adegoke21} for a proof of \eqref{eq.p951mmh}. Identity \eqref{eq.trp91mj} is a consequence of the Binet formula.
\end{proof}
\begin{lemma}\label{lem.dsfre22}
The following identities hold for integers $n$, $m$ and $r$:
\begin{equation}\label{eq.c1ni4vp}
\tau ^r u_{m - s}  = \tau ^m u_{r - s}  - q^{m - s} \tau ^s u_{r - m}\,,
\end{equation}
\begin{equation}
\sigma ^r u_{m - s}  = \sigma ^m u_{r - s}  - q^{m - s} \sigma ^s u_{r - m}\,,
\end{equation}
\begin{equation}\label{eq.gbxb02o}
\tau ^r u_{m - s} \Delta = \tau ^m v_{r - s}  - q^{m - s} \tau ^s v_{r - m}
\end{equation}
amd
\begin{equation}
\sigma ^r u_{m - s} \Delta =  - \sigma ^m v_{r - s}  + q^{m - s} \sigma ^s v_{r - m}\,.
\end{equation}
\end{lemma}
\begin{proof}
These are immediate consequences of the Binet formulas.
\end{proof}
\begin{lemma}
If $m$ and $n$ are integers, then \cite{Adegoke19}
\begin{align}
u_{n + m}  - q^m u_{n - m} = u_mv_n\label{eq.kb3hsvs},\\
v_{n + m}  - q^m v_{n - m}=\Delta^2u_m u_n \label{eq.nutjauf},\\
u_{n + m}  + q^m u_{n - m}= v_m u_n\label{eq.ciqxfvx} ,
\end{align}
and
\begin{equation}\label{eq.u8cxhf1}
v_{n + m}  + q^m v_{n - m}=v_m v_n.
\end{equation}
\end{lemma}
In the next result we give a generalization of Theorem~\ref{thm1}.
\begin{theorem}
If $r$, $n$ and $t$ are any integers, then
\begin{equation}\label{eq.eapqwai}
\begin{split}
&\sum_{j = 0}^n {q^{rj} w_{r(n - 2j) + t} }  = w_t \sum_{j = 0}^n {\frac{{v_r^j v_{r(n - j)} }}{{2^{j + 1} }}} \\
&= \frac{{w_{t + 1 + r(n + 1)}  - q^{r(n + 1)} w_{t + 1 - r(n + 1)} }}{{u_r \Delta ^2 }} - \frac{{q\left( {w_{t - 1 + r(n + 1)}  - q^{r(n + 1)} w_{t - 1 - r(n + 1)} } \right)}}{{u_r \Delta ^2 }}
\end{split}
\end{equation}
\end{theorem}
\begin{proof}
Set $(x,y)=(\tau^r,\sigma^r)$ and $(x,y)=(\sigma^r,\tau^r)$, in turn in~\eqref{eq.ttxukxg} and use the Binet formulas and Lemma~\ref{lem.w65xm59}. Note also the use of~\cite[Equation (3.16)]{horadam65}:
\begin{equation}
w_{r + s} + q^sw_{r - s} = v_s w_r.
\end{equation}
\end{proof}

\begin{corollary}
If $r$ and $n$ are any integers, then
\begin{equation}
\sum_{j = 0}^n q^{rj} u_{r(n - 2j)} = 0.
\end{equation}
\end{corollary}

\begin{corollary}
If $n$ is any integer, then
\begin{equation}
\sum_{j = 0}^n q^{j} v_{n - 2j} = \sum_{j=0}^n \left (\frac{p}{2}\right )^j v_{n-j} = 2 u_{n+1}.
\end{equation}
\end{corollary}

\begin{corollary}
If $r$, $n$ and $t$ are any integers, then
\begin{equation}
u_r \Delta^2 \mid w_{t + 1 + r(n + 1)} - q^{r(n + 1)} w_{t + 1 - r(n + 1)} - q\left( {w_{t - 1 + r(n + 1)} - q^{r(n + 1)} w_{t - 1 - r(n + 1)} } \right),
\end{equation}
provided both quantities are integers.
\end{corollary}
In particular, on account of~\eqref{eq.kb3hsvs} and~\eqref{eq.nutjauf}, we have
\begin{equation*}
u_r \mid u_{r(n + 1)} .
\end{equation*}
\begin{remark}
Doing the transformation $j\to n - j$ followed by $t\to t + rn$, we get an equivalent form of \eqref{eq.eapqwai} given by
\begin{equation}\label{eq.ci2o3yi}
\begin{split}
2\sum_{j = 0}^n {q^{r(n - j)} w_{2rj + t} } &=
\sum_{j = 0}^n {\left( {\frac{{v_r }}{2}} \right)^j \left( {w_{r(2n - j) + t}  + q^{r(n - j)} w_{rj + t} } \right)} \\
&= 2\frac{{w_{r(2n + 1) + t + 1} - qw_{r(2n + 1) + t - 1} - q^{r(n + 1)} \left( {w_{t - r + 1} - qw_{t - r - 1} } \right)}}{{u_r \Delta ^2 }}.
\end{split}
\end{equation}
\end{remark}

In the next theorem we present a generalization of Theorem~\ref{thm.ksm9y59}.

\begin{theorem}
Let $m$, $n$, $r$, $s$ and $t$ be any integers. Then
\begin{equation}
\begin{split}
&\sum_{j = 0}^n {( - 1)^j q^{(m - s)j} u_{r - s}^{n - j} u_{r - m}^j w_{(s - m)j + mn + t} } \\
&\qquad = \sum_{j = 0}^n {\frac{ u_{m - s}^j}{{2^{j + 1} }}\left( {u_{r - s}^{n - j} w_{(r - m)j + mn + t}  + ( - 1)^{n - j} q^{(m - s)(n - j)} u_{r - m}^{n - j} w_{s(n - j) + t + rj} } \right)}\\
&\qquad = \frac{{u_{r - s}^{n + 2} w_{mn + t}  + u_{r - s}^{n + 1} u_{r - m} w_{mn + m + t - s} }}{{u_{r - s}^2  + q^{m - s} u_{r - m}^2  + u_{r - s} u_{r - m} v_{m - s} }}\\
&\qquad\qquad + \frac{{( - 1)^n u_{r - m}^{n + 1} \left( {q^{(m - s)(n + 1) + m} u_{r - s} w_{sn + s + t - m}  + q^{(m - s)(n + 2) + s} u_{r - m} w_{sn + t} } \right)}}{{q^m u_{r - s}^2  + q^{2m - s} u_{r - m}^2  + q^m u_{r - s} u_{r - m} v_{m - s} }}.
\end{split}
\end{equation}
\end{theorem}
\begin{proof}
Set $(x,y)=(\tau^mu_{r - s},-q^{m - s}\tau^su_{r - m})$ and $(x,y)=(\sigma^mu_{r - s},-q^{m - s}\sigma^su_{r - m})$, in turn, in~\eqref{eq.ttxukxg}. Multiply through the $\tau$ equation by $\tau^t$ and the $\sigma$ equation by $\sigma^t$. Use the Binet formula and Lemma~\ref{lem.dsfre22}.
\end{proof}

\begin{corollary}
Let $m$, $n$, $r$, $s$ and $t$ be integers such that $t$ is non-negative and \mbox{$r\ge m\ge s\ge 0$}. Let
\begin{equation*}
X=X(m,n,r,s,t):=q^m u_{r - s}^2  + q^{2m - s} u_{r - m}^2  + q^m u_{r - s} u_{r - m} v_{m - s}
\end{equation*}
and
\begin{equation*}
\begin{split}
Y=Y(m,n,r,s,t)&:=q^mu_{r - s}^{n + 2} w_{mn + t}  + q^mu_{r - s}^{n + 1} u_{r - m} w_{mn + m + t - s}\\
&\qquad + ( - 1)^n u_{r - m}^{n + 1} \left( {q^{(m - s)(n + 1) + m} u_{r - s} w_{sn + s + t - m}  + q^{(m - s)(n + 2) + s} u_{r - m} w_{sn + t} } \right).
\end{split}
\end{equation*}
Then
\begin{equation*}
X\mid Y,
\end{equation*}
provided that the Horadam sequence parameters $p$, $q$, $a$ and $b$ are integers.
\end{corollary}

The next set of results generalzes Theorem~\ref{thm.wtn85im}.
\begin{theorem}
If $n$, $r$ and $t$ are any integers, then
\begin{equation}\label{eq.hrnx7vm}
\begin{split}
\sum_{j = 0}^{2n} {( - 1)^j q^{rj} w_{2r(n - j) + t} }  &= \frac{{w_t }}{2}\sum_{j = 0}^n {\left( {\frac{{u_r^2 \Delta ^2 }}{4}} \right)^j v_{2r(n - j)} }  + \frac{{w_t }}{{u_r }}\sum_{j = 1}^n {\left( {\frac{{u_r^2 \Delta ^2 }}{4}} \right)^j u_{r(2n - 2j + 1)} } \\
& = \frac{{w_t v_{r(2n + 1)} }}{{v_r }}
\end{split}
\end{equation}
and
\begin{equation}\label{eq.za08w0z}
\begin{split}
\sum_{j = 0}^{2n - 1} {( - 1)^j q^{rj} w_{r(2n - 1 - 2j) + t} }  &= \frac{{w_{t + 1}  - qw_{t - 1} }}{2}\sum_{j = 0}^{n - 1} {\left( {\frac{{u_r^2 \Delta ^2 }}{4}} \right)^j u_{r(2n - 2j - 1)} } \\
&\qquad + \frac{{w_{t + 1}  - qw_{t - 1} }}{{u_r \Delta ^2 }}\sum_{j = 1}^n {\left( {\frac{{u_r^2 \Delta ^2 }}{4}} \right)^j v_{r(2n - 2j)} } \\
& = \frac{{w_{t + 2rn}  - q^{2rn} w_{t - 2rn} }}{{v_r }}.
\end{split}
\end{equation}
\end{theorem}
\begin{proof}
As the steps in the proofs are clear, we omit the details. Choose $(x,y)=(\alpha^r,-\beta^r)$ and $(x,y)=(\beta^r,-\alpha^r)$,
in turn, and multiply through by $\alpha^t$ and $\beta^t$, respectively. Combine, using the Binet formula and Lemma \ref{lem.w65xm59}.
The proof of~\eqref{eq.za08w0z} is similar.
\end{proof}

\begin{corollary}
Let $m$ be any positive odd integer and $n$ any positive even integer. Let $r$ and $t$ be any non-negative integers. Then
\begin{align*}
&v_r \mid v_{rm},\\
&v_r \mid w_{t + rn}  - q^{rn} w_{t - rn} ;
\end{align*}
provided that the Horadam sequence parameters $p$, $q$, $a$ and $b$ are integers. In particular, $v_r\mid u_{rn}$.
\end{corollary}

\begin{corollary}
If $r$, $n$ and $t$ are any integers, then
\begin{equation}
\sum_{j = 0}^{2n} (- 1)^j q^{rj} u_{2r(n - j)} = 0,
\end{equation}
\begin{equation}
\sum_{j = 0}^{2n - 1} (- 1)^j q^{rj} v_{r(2n - 1 - 2j)} = 0,
\end{equation}
\begin{equation}
\begin{split}
\sum_{j = 0}^{2n - 1} (- 1)^j q^{rj} u_{r(2n - 1 - 2j) + t} &= \sum_{j = 0}^{n - 1} \left( {\frac{{u_r^2 \Delta ^2 }}{4}} \right)^j u_{r(2n - 2j - 1)} + \frac{2}{{u_r \Delta ^2 }}\sum_{j = 1}^n {\left( {\frac{{u_r^2 \Delta ^2 }}{4}} \right)^j v_{r(2n - 2j)} } \\
&= 2\frac{{u_{2rn} }}{{v_r }},
\end{split}
\end{equation}
and
\begin{equation}
\begin{split}
\sum_{j = 0}^{2n} (- 1)^j q^{rj} v_{2r(n - j)} &= \sum_{j = 0}^n {\left( {\frac{{u_r^2 \Delta ^2 }}{4}} \right)^j v_{2r(n - j)} }
+ \frac{2}{{u_r }}\sum_{j = 1}^n {\left( {\frac{{u_r^2 \Delta ^2 }}{4}} \right)^j u_{r(2n - 2j + 1)} } \\
&= \frac{{2v_{r(2n + 1)} }}{{v_r }}.
\end{split}
\end{equation}
\end{corollary}

\end{document}